\documentclass{amsart}

\usepackage{hyperref}
\hypersetup{
    colorlinks=true,
    linkcolor=black!50,
    urlcolor=pink,
    citecolor=blue
}
\usepackage{amssymb}
\usepackage{amsmath}
\usepackage{amsthm}
\usepackage{enumerate}
\usepackage{graphicx}
\usepackage{wrapfig}
\usepackage{color}
\usepackage{subcaption}
\usepackage{float}

\usepackage{tikz}
\usetikzlibrary{hobby}
\usetikzlibrary{bbox}

\theoremstyle{plain}

\makeatletter
\newtheorem*{rep@theorem}{\rep@title}
\newcommand{\newreptheorem}[2]{%
\newenvironment{rep#1}[1]{%
 \def\rep@title{#2 \ref{##1}}%
 \begin{rep@theorem}}%
 {\end{rep@theorem}}}
\makeatother

\newtheorem{theorem}{Theorem}[section]
\newreptheorem{theorem}{Theorem}
\newtheorem{proposition}[theorem]{Proposition}

\newtheorem{lemma}[theorem]{Lemma}
\newtheorem{corollary}[theorem]{Corollary}
\newtheorem{question}[theorem]{Question}

\newtheorem{problem}[theorem]{Problem}

\theoremstyle{definition}
\newtheorem{definition}[theorem]{Definition}
\newtheorem{example}{Example}[section]

\theoremstyle{remark}
\newtheorem*{remark}{Remark}

\newcommand{\mC}{\mathbb{C}}
\newcommand{\mR}{\mathbb{R}}

\newcommand{\mZ}{\mathbb{Z}}
\newcommand{\mD}{\mathbb{D}}
\newcommand{\mT}{\mathbb{T}}

\newcommand{\hC}{\widehat{\mathbb{C}}}

\renewcommand {\tilde} {\widetilde}

\renewcommand{\Im}{\mathrm{Im\,}}

\begin{document}

\title{Rational lemniscates and the matching problem}

\date{April 26, 2025}

\author[K. Lazebnik]{Kirill Lazebnik}
\address{Mathematical Sciences, University of Texas at Dallas, Richardson, Texas 75080-3021, USA.}
\email{Kirill.Lazebnik@UTDallas.edu}

\author[P.-O. Paris\'e]{Pierre-Olivier Paris\'e}
\address{D\'epartement de math\'ematiques et d'informatique, Universit\'e du Qu\'ebec \`a Trois-Rivi\`eres, Trois-Rivi\`eres, Qu\'ebec G8Z 4M3, Canada}
\email{pierre-olivier.parise@uqtr.ca}

\author[M. Younsi]{Malik Younsi}
\address{Department of Mathematics, University of Hawaii Manoa, Honolulu, HI 96822, USA.}
\email{malik.younsi@gmail.com}

\thanks{K.L was supported by NSF Grant DMS-2246876. 
P.-O.P. was supported by an NSERC postdoctoral scholarship. M.Y. was supported by NSF Grant DMS-2350530.}


\begin{abstract}
The matching problem for a given Jordan curve in the complex plane asks to find two nonconstant functions, one analytic in the bounded complementary component of the curve and the other analytic in the unbounded complementary component of the curve, which are continuous up to the curve and complex conjugate to each other on the curve. We prove that there exist Jordan curves of any Hausdorff dimension between $1$ and $2$ for which the matching problem has a solution. This answers a question of Ebenfelt--Khavinson--Shapiro and provides the first examples of solutions to the matching problem other than rational lemniscates. Our approach relies on conformal welding and harmonic measure. We also obtain new examples of Jordan curves for which the matching problem has no solution, and give a characterization of the subsets of the Riemann sphere that are rational lemniscates in terms of harmonic measure.
\end{abstract}

\maketitle

\section{Introduction}
\label{sec1}

Let $\Gamma$ be a Jordan curve in the complex plane $\mC$, and denote by $\Omega$ and $\Omega^*$ the bounded and unbounded components of the complement of $\Gamma$ in the Riemann sphere $\hC$. Let $A(\Omega)$ be the space of continuous functions on $\overline{\Omega}$ that are analytic in $\Omega$, and similarly let $A(\Omega^*)$ be the corresponding space on $\Omega^*$ with the additional requirement that every function in $A(\Omega^*)$ vanishes at infinity. This article is devoted to the study of the so-called \textit{matching problem} for $\Gamma$:

\begin{problem}[Matching Problem]
Do there exist non-constant functions $f \in A(\Omega)$ and $g \in A(\Omega^*)$ such that $f=\overline{g}$ on $\Gamma$?
\end{problem}
If such functions $f$ and $g$ exist, then we shall say that the curve $\Gamma$ is \textit{MP-solvable} and will refer to the pair $(f,g)$ as a \textit{matching pair} for $\Gamma$.

The matching problem was introduced by Ebenfelt--Khavinson--Shapiro \cite{EbenfeltKhavinsonShapiro2001}, motivated by the study of double-layer potentials and the kernel of the Neumann--Poincar\'e operator in connection to the Dirichlet problem. See also \cite[Section 3]{Khavinson2018}.

In \cite{EbenfeltKhavinsonShapiro2001}, the authors gave examples of MP-solvable Jordan curves arising from rational maps.

\begin{example}[Ebenfelt--Khavinson--Shapiro]
\label{lemniscate}
Let $r$ be a rational map, and suppose that for some $c>0$ the lemniscate $\Gamma:=\{z \in \hC:|r(z)|=c\}$ is a Jordan curve. As before denote by $\Omega$ and $\Omega^*$ the bounded and unbounded complementary components of $\Gamma$ respectively. Suppose in addition that
\begin{enumerate}[(i)]
\item $r$ has no pole in $\Omega$;
\item $r$ has no zero in $\Omega^*$;
\item $r(\infty)=\infty.$
\end{enumerate}
Then the pair $(f,g)$ with $f:=r, g:=c^2/r$ is a matching pair for $\Gamma$. In particular $\Gamma$ is MP-solvable.
\end{example}

In \cite{EbenfeltKhavinsonShapiro2001}, the authors mention that rational lemniscates are the only known examples of MP-solvable Jordan curves. In particular all known examples are analytic curves, which naturally leads to the following question (\cite[Remark 1.5]{EbenfeltKhavinsonShapiro2001}).

\begin{question}[Ebenfelt--Khavinson--Shapiro]
\label{mainq}
How much (if any) regularity of $\Gamma$ is forced by the existence of a matching pair?
\end{question}

See also \cite[Problem 3.30]{HaymanLingham2019}.

With Question \ref{mainq} in mind, let us state our first main result, which provides new examples of MP-solvable Jordan curves. In the following $\mathcal{H}^1$ denotes the one-dimensional Hausdorff measure.

\begin{theorem}
\label{mainthm}
Let $\Gamma$ be a Jordan curve in $\mC$. Denote by $T_\Gamma$ the set of tangent points of $\Gamma$. If $\mathcal{H}^1(T_\Gamma)=0$, then $\Gamma$ is MP-solvable.
\end{theorem}

It is not difficult to construct curves as in Theorem \ref{mainthm} with any prescribed Hausdorff dimension between one and two, or even positive area. This shows that no regularity at all is forced by the existence of a matching pair, answering Question \ref{mainq}.

\begin{corollary}
\label{mainc}
For any $\alpha \in [1,2]$, there exists an MP-solvable Jordan curve $\Gamma$ with $\dim_H(\Gamma)=\alpha$, where $\dim_H$ denotes Hausdorff dimension. If $\alpha=2$ then $\Gamma$ can be constructed so that its area is positive.
\end{corollary}

The proof of Theorem \ref{mainthm} is based on a new characterization of MP-solvable curves in terms of conformal welding. As we shall see, the work of Bishop--Carleson--Garnett--Jones \cite{BishopCarlesonGarnettJones1989} on harmonic measure and the work of Browder--Wermer \cite{BrowderWermer1963} on Dirichlet algebras play a fundamental role.

We now discuss examples of Jordan curves that are not MP-solvable. The first examples of such curves were obtained by Ebenfelt, Khavinson and Shapiro in \cite{EbenfeltKhavinsonShapiro2001}.

\begin{theorem}[Ebenfelt--Khavinson--Shapiro]
\label{nomatching}
Let $r$ be a rational map of degree at least $2$ which is conformal in a neighborhood of the closed unit disk $\overline{\mD}$. Then $\Gamma:=r(\partial \mD)$ is not MP-solvable.
\end{theorem}

Note that Example \ref{lemniscate} shows that MP-solvable curves can be obtained by taking preimages of the unit circle under rational maps, while Theorem \ref{nomatching} shows that curves that are \textit{not} MP-solvable can be obtained by taking \textit{forward} images of the unit circle under rational maps. Thus even among algebraic curves the existence of matching pairs is a delicate problem.

Our second main result provides new examples of Jordan curves for which the matching problem has no solution.

\begin{theorem}
\label{mainthm2}
Let $\Gamma$ be a Jordan curve in $\mC$. Suppose that there exists an open half-plane $\mathbb{H}$ such that $\Gamma \subset \overline{\mathbb{H}}$ and $\Gamma \cap \partial \mathbb{H}$ is a line segment. Then $\Gamma$ is not MP-solvable.
\end{theorem}

In particular triangles, squares and more generally regular polygons are not MP-solvable.

It easily follows from either Theorem \ref{nomatching} or Theorem \ref{mainthm2} that the collection of Jordan curves that are not MP-solvable is dense in the set of all Jordan curves with respect to the Hausdorff distance. On the other hand, Example \ref{lemniscate} combined with Hilbert's lemniscate theorem shows that the collection of MP-solvable Jordan curves is also dense. See Corollary \ref{corollary:dense}.

The remainder of our results deal with rational lemniscates, motivated by Example \ref{lemniscate}. Recall that a \textit{rational lemniscate} is a subset of the Riemann sphere $\hC$ of the form $L_r(c):=\{z \in \hC : |r(z)|=c\}$ for some rational map $r$ and some constant $0<c<\infty$. For simplicity, throughout the rest of the paper we will assume $c=1$, rescaling $r$ if necessary. In this case we write $L_r:=L_r(1)$.

As noted in \cite{BishopEremenkoLazebnik2025}, the topology of $L_r$ can be quite complicated (see Figure \ref{fig:lemGraph}), it may have several connected components, and even connected components that are not Jordan curves (this occurs whenever the lemniscate contains a critical point). For the matching problem, however, one needs $L_r$ to be a Jordan curve. Searching for MP-solvable Jordan curves that are not rational lemniscates naturally leads to the following question.

\begin{question}
\label{q1}
Given a Jordan curve $\Gamma$, how can one tell whether $\Gamma$ is equal to $L_r$ for some rational map $r$? More generally, which subsets of the Riemann sphere are rational lemniscates?
\end{question}

Our answer to Question \ref{q1} is roughly that there are three conditions (see (\ref{top_cond})-(\ref{alg_cond}) in Theorem \ref{answerq1} below) that a rational lemniscate must necessarily satisfy, and that these three conditions for a given subset of $\hC$ are in fact also sufficient to guarantee that the subset is a rational lemniscate. The three conditions fall into three categories: topological, analytic and algebraic. We first discuss the topological condition, which was introduced in \cite{BishopEremenkoLazebnik2025} to describe the topology of rational lemniscates.

\begin{definition}\label{embeddedgraph}
A  \emph{lemniscate graph} is a set
$G\subset \hC$ so that there is a finite set
$V\subset G$ (called the \emph{vertices} of $G$), so that:
\begin{enumerate}
\item $G\setminus V$ has finitely many components
(these are called the \emph{edges} of $G$),
each of which is either a (closed) Jordan curve,
or else an (open) simple arc $\gamma$ satisfying
$\overline{\gamma}\setminus\gamma\subset V$.
\item The degree of each vertex is even and at least four,
where the \emph{degree} of a vertex $v$ is defined
as the number of edges $\gamma$ satisfying
$v\in\overline{\gamma}\setminus\gamma$, and we count
an edge $\gamma$ twice if $\{v\}=\overline{\gamma}\setminus\gamma$.
\end{enumerate}
The components of $\hC \setminus G$ are called the \emph{faces} of $G$, and a $2$-coloring of $G$ is an assignment of one of two colors to each face (we will use white and grey) so that any two faces sharing a common edge have different colors. See Figure \ref{fig:lemGraph}.
\end{definition}

\begin{figure}
    \centering
    \begin{tikzpicture}[bezier bounding box]
        \draw[white, fill=black!30] (-5,-2.5) rectangle (5,2.5);
        \draw[black, fill=white, line width=0.75pt] (-4.25,0) to[curve through = {(-1.5,1.75) (0,0) (1.5,-1.75) (4.25,0) (1.5, 1.75) (0,0) (-1.5, -1.75)}] (-4.25, 0);
        \draw[shift={(-2.75, 1)}, rotate=30,scale=1.75,  black, fill=black!30, line width=0.75] 
        (0, 0) .. controls (0.707, 0.707) and (0.707, -0.707) .. (0,0)
        .. controls (-0.707, 0.707) and (-0.707, -0.707) .. (0,0);
        \draw[shift={(-2, -0.35)}, scale=2, black, fill=black!30, line width=0.75pt] 
        (0.001, 0.001) .. controls +(1, 0) and +(0, 1) .. (0.001, 0.001);
        \draw[shift={(-2, -0.35)}, scale=2, black, fill=black!30, line width=0.75pt] 
        (0.001, 0.001) .. controls +(-0.5, 0.866) and +(-0.866,-0.5) .. (0.001, 0.001);
        \draw[shift={(-2, -0.35)}, scale=2, black, fill=black!30, line width=0.75pt] 
        (0.001, 0.001) .. controls +(-0.5, -0.866) and +(0.866,-0.5) .. (0.001, 0.001);
        \draw[black, fill=black!30, line width=0.75pt] (-3.25, -1) circle(0.125);
        \draw[black, fill=white, line width=0.75pt] (-1.8, -0.85) circle(0.2);
        \draw[black, fill=black!30, line width=0.75pt] (2.25, 0) circle(1);
       \draw[shift={(2.25, 0)}, rotate=30,scale=1.25,  black, fill=white, line width=0.75pt] 
        (0, 0) .. controls (0.866, 0.5) and (0.866, -0.5) .. (0,0)
        .. controls (-0.866, 0.5) and (-0.866, -0.5) .. (0,0);
        \draw[shift={(2.25, 0)}, rotate=-60,scale=1.25,  black, fill=white, line width=0.75pt] 
        (0, 0) .. controls (0.866, 0.5) and (0.866, -0.5) .. (0,0)
        .. controls (-0.866, 0.5) and (-0.866, -0.5) .. (0,0);
        \foreach \x in {0, 1, 2, 3, 4, 5, 6, 7, 8}{
            \draw[black, fill=black!30, line width=0.75pt] ({2.25 + 1.5*cos(\x * 360 / 9)}, {1.5*sin(\x * 360 / 9)}) circle(0.2);
        }
    \end{tikzpicture}
    \caption{An example of a lemniscate graph}
    \label{fig:lemGraph}
\end{figure}
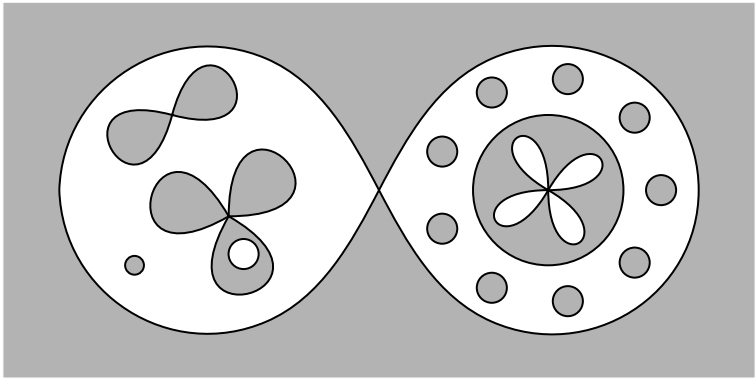

It is not difficult to prove that every rational lemniscate $L_r$ is a lemniscate graph, see \cite[Proposition 2.4]{BishopEremenkoLazebnik2025}. In this case each degree $d$ critical point of $r$ lying on $L_r$ is a vertex of degree $2(d+1)$. Choosing a coloring of the faces of $L_r$ corresponds to a choice of whether to color the components of $r^{-1}(\mathbb{D})$ white and the components of $r^{-1}(\mD^*)$ grey, or vice versa; we will follow the convention of coloring the components of $r^{-1}(\mathbb{D})$ white. Here $\mD^*:=\hC \setminus \overline{\mD}$.

Our answer to Question \ref{q1} is the following. Roughly, it says that a subset $G \subset \hC$ which is a lemniscate graph, and also satisfies an analytic condition (\ref{analyt_cond}) and an algebraic condition (\ref{alg_cond}), must be a rational lemniscate. These two conditions involve harmonic measure, denoted by $\omega$.

\begin{theorem}
\label{answerq1}
A set $G \subset \hC$ is a rational lemniscate if and only if
\begin{equation}\label{top_cond} G \textrm{ is a lemniscate graph, }
\end{equation}
and there exist points $(z_i)_{i=1}^m \subset \hC \setminus G$, not necessarily distinct, so that for any two faces $E$, $F$, we have
\begin{equation}\label{analyt_cond} \sum_{z_i \in E} \omega(B, z_i, E) = \sum_{z_i \in F} \omega(B, z_i, F) \emph{ for every Borel set } B\subset \partial E\cap \partial F\emph{, and }
\end{equation}
\begin{equation}\label{alg_cond} \sum_{z_i \in E} \omega(\Gamma, z_i, E) \in \mathbb{Z} \emph{ for every component } \Gamma \emph{ of } \partial E.
\end{equation}
If this is the case, then $G=L_r$ where the $z_i$ lying in the white faces of $G$ coincide with the zeros of $r$ (counted with multiplicities), and the $z_i$ lying in the grey faces of $G$ coincide with the poles of $r$  (counted with multiplicities).
\end{theorem}

As we shall see, the necessity in Theorem \ref{answerq1} is more straightforward; the converse indicates more surprisingly that three natural conditions (topological, analytic and algebraic) which must hold for any rational lemniscate are in fact sufficient for a set $G \subset \hC$ to be a rational lemniscate. Note also that in Theorem \ref{answerq1} no regularity at all is assumed for $G$; rather it is a consequence of the theorem that every edge of a lemniscate graph satisfying the analytic condition and the algebraic condition turns out to be analytic.

In the polynomial case, Theorem \ref{answerq1} reduces to the following.

\begin{theorem}
\label{polyanswerq1} A set $G \subset \mC$ is a polynomial lemniscate if and only if $G$ is a lemniscate graph such that each bounded component of $\hC \setminus G$ is simply connected, and there exist points $(z_i)_{i=1}^m$, not necessarily distinct, in the union of the bounded components of $\hC \setminus G$ so that for any bounded face $E$, and denoting the unbounded face by $F$, we have
\begin{equation*} \frac{1}{m}\sum_{z_i \in E} \omega(B, z_i, E) = \omega(B, \infty, F) \emph{ for every Borel set } B\subset \partial E\cap \partial F\emph{, and }
\end{equation*}
\begin{equation*} m\omega(\Gamma, \infty, F) \in \mathbb{Z} \emph{ for every component } \Gamma \emph{ of } G.
\end{equation*}
If this is the case, then $G=L_p$ where the $z_i$ lying in the bounded faces of $G$ coincide with the zeros of $p$ (counted with multiplicities).
\end{theorem}

In the recent work \cite{BishopEremenkoLazebnik2025}, Bishop, Eremenko and the second author proved that every lemniscate graph can be approximated in a strong sense by a rational lemniscate. In other words, the topological condition in Theorem \ref{answerq1} for a given set $G \subset \hC$ guarantees that $G$ is \textit{approximately} a rational lemniscate. Theorem \ref{answerq1} therefore addresses the question of which additional conditions must be imposed in order for $G$ to be \textit{exactly} a rational lemniscate.

Polynomial lemniscates have been studied extensively in the literature, see e.g. \cite{BishopEremenkoLazebnik2025} and the references therein. The conformal properties of rational lemniscates were studied by Fortier-Bourque and the third author in \cite{FortierBourqueYounsi2015} and in \cite{Younsi2016}. We also mention the work of Pouliasis and Ransford in \cite{PouliasisRansford2016} on the harmonic measure and capacity of rational lemniscates, as well as the work of Lerario and Lundberg in \cite{LerarioLundberg2016} on the geometry of random rational lemniscates. Lastly, it was recently observed by Viklund and Witt-Nystr\"{o}m in \cite{WittViklund} that rational lemniscates arise as stationary solutions to some competitive Hele-Shaw flow problem.

The rest of the article is structured as follows. Section \ref{sec2} contains various preliminaries that will be needed throughout the paper. In Section \ref{sec3} we obtain a new characterization of MP-solvable curves in terms of conformal welding, allowing us to prove Theorem \ref{mainthm}. In Section \ref{sec4} we construct MP-solvable curves of any prescribed dimension, thereby proving Corollary \ref{mainc}. Section \ref{sec5} is devoted to examples of Jordan curves for which the matching problem has no solution, namely the proof of Theorem \ref{mainthm2}. In Section \ref{sec6} we prove Theorem \ref{answerq1}. Lastly, Section \ref{sec7} contains a discussion of some open problems related to the matching problem and to rational lemniscates.

\section{Preliminaries}
\label{sec2}

\subsection{Notation}

The following notation will be used throughout the paper. We denote the complex plane by $\mC$, the Riemann sphere (extended complex plane) by $\hC$, the open unit disk by $\mD$ and the unit circle by $\mT$. We also let $\mD^*:=\hC \setminus \overline{\mD}$.

Let $A$ be a subset of $\mC$. For $s \geq 0$ and $0<\delta \leq \infty$, we define
$$\mathcal{H}_{\delta}^{s}(A):= \inf \left\{ \sum_j \operatorname{diam}(A_j)^s : A \subset \bigcup_j A_j, A_j \subset \mC, \operatorname{diam}(A_j) \leq \delta \right\}.$$
Here $\operatorname{diam}(A_j)$ denotes the diameter of the set $A_j$:
$$\operatorname{diam}(A_j):=\sup_{z,w \in A_j} |z-w|.$$
The $s$-\textit{dimensional Hausdorff measure} of $A$ is
$$\mathcal{H}^s(A) := \sup_{\delta>0} \mathcal{H}_{\delta}^s(A) = \lim_{\delta \to 0} \mathcal{H}_{\delta}^s(A).$$
The \textit{Hausdorff dimension} of $A$ is the unique positive number $\operatorname{dim}_H(A)$ such that

\begin{displaymath}
\mathcal{H}^{s}(A) = \left\{ \begin{array}{ll}
\infty & \textrm{if $s < \dim_H(A) $}\\
0 & \textrm{if $s>\dim_H(A)$}.\\
\end{array} \right.
\end{displaymath}

\subsection{Harmonic measure and Green's function}

Let $\Omega \subset \hC$ be a domain with finitely many boundary components, each containing more than one point.

For $w \in \Omega$, we denote by $\omega(\cdot,w,\Omega)$ the harmonic measure for $\Omega$ with respect to $w$. We will need the following well-known property of harmonic measure.

\begin{proposition}[Conformal invariance of harmonic measure]
Let $f:\Omega \to \Omega'$ be a conformal map and let $B_1$ and $B_2$ be Borel subsets of $\partial \Omega$ and $\partial \Omega'$ respectively. If $f$ extends to a homeomorphism of $ B_1$ onto $ B_2$, then
    $$\omega(B_1,w,\Omega) = \omega(B_2,f(w),\Omega') \qquad (w \in \Omega).$$
\end{proposition}
See e.g. \cite[Theorem 4.3.8]{Ransford1995}.

Harmonic measure is closely related to another important conformal invariant, the Green's function.

\begin{definition}
The Green's function for $\Omega$ is the unique function $g_\Omega:\Omega \times \Omega \to (0,\infty]$ such that for each $w \in \Omega$:

\begin{enumerate}[(i)]
\item $g_\Omega(\cdot,w)$ is harmonic on $\Omega \setminus \{w\}$ and bounded outside each neighborhood of $w$;
\item $g_\Omega(w,w)=\infty$, and as $z \to w$,
\begin{displaymath}
g_\Omega(z,w) = \left\{ \begin{array}{ll}
\log{|z|}+O(1) & \textrm{if $w=\infty $}\\
-\log{|z-w|}+O(1) & \textrm{if $w \neq \infty$};\\
\end{array} \right.
\end{displaymath}
\item $g_\Omega(z,w) \to 0$ as $z \to \zeta$, for all $\zeta \in \partial \Omega$.
\end{enumerate}
\end{definition}

The Green's function also satisfies a conformal invariance property.

\begin{proposition}[Conformal invariance of Green's function]
If $f:\Omega \to \Omega'$ is conformal, then
$$g_{\Omega'}(f(z),f(w)) = g_\Omega(z,w) \qquad (z,w \in \Omega).$$
\end{proposition}
See e.g. \cite[Theorem 4.4.4]{Ransford1995}.

We now turn our attention to a more specific class of domains $\Omega \subset \hC$ related to the notion of lemniscate graph from the introduction.

\begin{definition}
\label{lemniscatedomain}
A domain $\Omega \subset \hC$ is called a \textit{lemniscate domain} if its boundary $\partial \Omega$ is a lemniscate graph. In the special case where $\partial \Omega$ consists of finitely many pairwise disjoint analytic Jordan curves, then we call $\Omega$ an \textit{analytic Jordan domain}.
\end{definition}

\begin{remark}
By definition every analytic Jordan domain is a lemniscate domain, but not conversely. On the other hand every lemniscate domain $\Omega$ can be mapped conformally onto an analytic Jordan domain $\Omega'$. This follows by repeated application of the Riemann mapping theorem, see e.g. \cite[Chapter 15, Theorem 2.1]{Conway1995}. Moreover any such conformal map $f:\Omega \to \Omega'$ has a continuous and injective extension to $\overline{\Omega} \setminus V$, where $V \subset \partial \Omega$ is the set of vertices of the lemniscate graph $\partial \Omega$. See \cite[Chapter 15, Theorem 3.6]{Conway1995}.
\end{remark}

The following proposition relates harmonic measure and the Green's function for analytic lemniscate domains.

\begin{proposition}
\label{extension}
Let $\Omega$ be a lemniscate domain. Suppose that every edge of $\partial \Omega$ is analytic. Then the Green's function for $\Omega$ extends to be harmonic on a neighborhood of $\partial \Omega \setminus V$, where $V \subset \partial \Omega$ is the set of vertices. Moreover we have, for $w \in \Omega$,
$$-\frac{\partial g_\Omega (\zeta , w )}{\partial n_\zeta} > 0 \qquad (\zeta \in \partial \Omega \setminus V),$$
where $n_\zeta$ is the unit outer normal at $\zeta$, and
$$d\omega(\zeta,w,\Omega) = -\frac{1}{2\pi} \frac{\partial g_\Omega (\zeta , w)}{\partial n_\zeta} |d\zeta|$$
as measures on $\partial \Omega \setminus V$.
\end{proposition}

\begin{proof}
This is well-known in the case where $\Omega$ is an analytic Jordan domain, see e.g. \cite[Chapter II, Theorem 2.5]{GarnettMarshall2005} and \cite[Chapter II, Corollary 2.6]{GarnettMarshall2005}. In the general case, let $f:\Omega \to \Omega'$ be a conformal map from the lemniscate domain $\Omega$ onto an analytic Jordan domain $\Omega'$, as in the remark following Definition \ref{lemniscatedomain}. Then for each $\zeta \in \partial \Omega \setminus V$, we have that $f$ extends analytically to a neighborhood of $\zeta$, by the Schwarz reflection principle for analytic arcs (\cite[Theorem 10.37]{Zakeri2021}). The result then follows from the conformal invariance of both harmonic measure and the Green's function.
\end{proof}

\subsection{Proper holomorphic mappings}

We will also need the notion of proper holomorphic mapping from a domain $\Omega \subset \hC$ onto either the unit disk or the complement of the closed unit disk. For simplicity let us assume that the target domain is $\mD^*$; with suitable modifications all the results in this subsection remain valid if we replace $\mD^*$ by $\mD$, as is easily seen by composing with $z \mapsto 1/z$.

\begin{definition}
Let $\Omega \subset \hC$ be a domain. A holomorphic function $B:\Omega \to \mD^*$ is called a \textit{proper holomorphic mapping} if $B^{-1}(K)$ is a compact subset of $\Omega$ whenever $K$ is a compact subset of $\mD^*$, or equivalently if $|B(z)| \to 1$ whenever $z \to \partial \Omega$.

\end{definition}

\begin{proposition}
\label{properextension}
Let $B:\Omega \to \mD^*$ be a proper holomorphic mapping.

\begin{enumerate}[(i)]
\item If $\Omega$ is an analytic Jordan domain, then $B$ extends to be analytic on a neighborhood of $\overline{\Omega}$. Moreover for each boundary curve $\Gamma \subset \partial \Omega$ the restriction $B:\Gamma \to \mT$ is a finite degree covering map.
\item If $\Omega$ is a lemniscate domain, then $B$ extends continuously to $\partial \Omega \setminus V$, where $V \subset \partial \Omega$ is the set of vertices of the lemniscate graph $\partial \Omega$. Moreover for each edge $\gamma \subset \partial \Omega \setminus V$, the restriction $B:\gamma \to B(\gamma)$ is a finite degree covering map.
\end{enumerate}
\end{proposition}

\begin{proof}
Suppose first that $\Omega$ is an analytic Jordan domain. Let $\Gamma$ be one of the analytic boundary curves. Let $U$ be the complementary component of $\Gamma$ containing $\Omega$, and let $F:\mD^* \to U$ be a conformal map. Since $\partial U = \Gamma$ is analytic, the map $F$ extends to a conformal map from a neighborhood of $\mT$ onto a neighborhood of $\Gamma$. Now, the function $B_1:=B \circ F$ is holomorphic on $\{z \in \mC: 1<|z|<r\}$ for some $r>1$, and $|B_1(z)| \to 1$ as $|z| \to 1$. It follows that $B_1$ extends to be analytic on a neighborhood of $\mT$, by the Schwarz reflection principle. Equivalently $B = B_1 \circ F^{-1}$ extends to be analytic on a neighborhood of $\Gamma$. Then $B$ must map $\Gamma$ onto $\mT$ and has no critical point on $\Gamma$, by local normal form. Also, it is easy to see that the extension of $B$ remains proper. It follows from \cite[Corollary 12.41]{Zakeri2021} that $B:\Gamma \to \mT$ is a finite degree covering map, as required. This proves (i).
\\

Suppose now that $\Omega$ is a lemniscate domain. Let $f:\Omega \to \Omega'$ be a conformal map onto an analytic Jordan domain $\Omega'$, as in the remark following Definition \ref{lemniscatedomain}. Recall that $f$ extends to be continuous and injective on $\overline{\Omega} \setminus V$. Since $B \circ f^{-1}: \Omega' \to \mD^*$ is a proper holomorphic mapping, it extends to be analytic on a neighborhood of $\overline{\Omega'}$, by (i). It follows that $B= (B \circ f^{-1}) \circ f$ extends continuously to $\partial \Omega \setminus V$. Also, if $\gamma \subset \partial \Omega \setminus V$ is an edge, then $f(\gamma)$ is a subarc in one of the boundary curves of $\Omega'$, and the restriction $f:\gamma \to f(\gamma)$ is injective. It follows from (i) that the restriction $B \circ f^{-1}: f(\gamma) \to B(\gamma)$ is a finite degree covering map, and the same is true for $B= (B \circ f^{-1}) \circ f: \gamma \to B(\gamma)$, as required.
\end{proof}

We will also need the following proposition relating proper holomorphic mappings and the Green's function.

\begin{proposition}
\label{proper1}
Let $\Omega \subset \hC$ be a lemniscate domain and let $B:\Omega \to \mD^*$ be a proper holomorphic mapping. If $w_1, \dots, w_n \in \Omega$ are the poles of $B$ repeated according to multiplicities, then
$$
\log{|B(z)|} = \sum_{k=1}^n g_{\Omega}(z,w_k) \qquad (z \in \Omega).
$$
\end{proposition}

\begin{proof}
We may assume $\Omega$ is bounded. Define
$$
u(z):=\log{|B(z)|} - \sum_{k=1}^n g_{\Omega}(z,w_k) \qquad (z \in \Omega).
$$
Then $u$ is harmonic on $\Omega \setminus \{w_1, \dots, w_n\}$ and bounded in a neighborhood of each $w_j$. Moreover, since $B$ is proper, we have that $u(z) \to 0$ as $z$ approaches $\partial \Omega$. By the Lindel{\"o}f's Maximum Principle (see \cite[Chapter I, Lemma 1.1]{GarnettMarshall2005}) applied to $u$ and $-u$, we get that $u$ is identically zero in $\Omega$, as required.
\end{proof}

The existence of proper holomorphic mappings with prescribed poles as in Proposition \ref{proper1} is related to a condition involving harmonic measure. This is where the algebraic condition in Theorem \ref{answerq1} comes from.

\begin{proposition}
\label{proper2}
Let $\Omega \subset \hC$ be a lemniscate domain with boundary components $\Gamma_1, \dots, \Gamma_N$, and let $w_1, \dots, w_n$ be points in $\Omega$, not necessarily distinct. Then there exists a proper holomorphic mapping $B:\Omega \to \mD^*$ with poles precisely at $w_1, \dots, w_n$ if and only if for all $j \in \{1, \dots, N\}$ we have
\begin{equation}
\label{eqhm}
\sum_{k=1}^n \omega(\Gamma_j, w_k, \Omega) \in \mZ.
\end{equation}

\end{proposition}

\begin{proof}
We may assume $\Omega$ is bounded.

Suppose that Equation (\ref{eqhm}) holds, and first assume that $\Omega$ is an analytic Jordan domain. Consider the function
$$
v(z):=\sum_{k=1}^n g_{\Omega}(z,w_k) \qquad (z \in \Omega).
$$
Then $v$ is harmonic on $\Omega \setminus \{w_1, \dots, w_n\}$ and has a harmonic extension to a neighborhood of $\partial \Omega$, by Proposition \ref{extension}. The harmonic conjugate $\tilde{v}(z)$ is well-defined locally, but when we analytically continue along a given boundary curve $\Gamma_j$, the value of $\tilde{v}(z)$ changes by
$$
\int_{\Gamma_j} \frac{\partial v}{\partial n_\zeta} \, |d\zeta|.
$$
Using Equation (\ref{eqhm}) and Proposition \ref{extension}, it easily follows that the value of $\tilde{v}(z)$ changes by an integer multiple of $2\pi$ under harmonic continuation. Thus
$$
B(z):=\exp{\left(v(z) + i \tilde{v}(z)\right)} \qquad (z \in \Omega)
$$
is a well-defined holomorphic function, and it is readily checked that $B$ is a proper holomorphic mapping of $\Omega$ onto $\mD^*$ with poles precisely at $w_1, \dots, w_n$. This proves the desired implication, but only in the special case where $\Omega$ is an analytic Jordan domain. Now, if $\Omega$ is a general lemniscate domain, then consider a conformal map $f:\Omega \to \Omega'$ onto an analytic Jordan domain $\Omega'$, as in the remark following Definition \ref{lemniscatedomain}. Then $f$ has a continuous and injective extension to $\overline{\Omega} \setminus V$, where $V \subset \partial \Omega$ is the set of vertices. Since $V$ is a finite set and therefore has zero harmonic measure, it follows from Equation (\ref{eqhm}) and the conformal invariance of harmonic measure that
$$
\sum_{k=1}^n \omega(\Gamma', f(w_k), \Omega') \in \mZ
$$
for every boundary component $\Gamma'$ of $\Omega'$. By the analytic case, there exists a proper holomorphic mapping $\tilde{B}:\Omega' \to \mD^*$ with poles precisely at $f(w_1), \dots, f(w_n)$. Then $B:=\tilde{B} \circ f$ gives a proper holomorphic mapping of $\Omega$ onto $\mD^*$ with poles precisely at $w_1, \dots, w_n$, as required.
\\

Conversely, suppose that there exists a mapping $B$ with the specified properties. Composing with a conformal map if necessary, we may assume once again that $\Omega$ is an analytic Jordan domain. Recall that if $\Gamma$ is a boundary component of $\Omega$ and if $h$ is a smooth function defined in a neighborhood of $\Gamma$, then the \textit{period} of $h$ along $\Gamma$ is
$$
\mathrm{per} (h , \Gamma ) := \int_{\Gamma} \frac{\partial h}{\partial n_\zeta}\, |d\zeta| .
$$
Now, by Proposition \ref{proper1}, we have
$$
\log{|B(z)|} = \sum_{k=1}^n g_{\Omega}(z,w_k) \qquad (z \in \Omega).
$$
But $B$ is single-valued, so that for each $j \in \{1, \dots, N\}$, we have
\begin{align*}
\sum_{k=1}^n \omega(\Gamma_j,w_k,\Omega) &= \sum_{k=1}^n \left(- \frac{1}{2\pi} \operatorname{per}(g_{\Omega}(z,w_k), \Gamma_j)\right)\\
&= - \frac{1}{2\pi} \operatorname{per}\left(\sum_{k=1}^n g_{\Omega}(z,w_k), \Gamma_j\right)\\
&= - \frac{1}{2\pi} \operatorname{per}(\log{|B(z)|}, \Gamma_j) \in \mathbb{Z},
\end{align*}
as required.
\end{proof}

\subsection{Jordan curves and conformal welding}

A crucial ingredient in the proof of Theorem \ref{mainthm} is the notion of conformal welding.

Let $\Gamma$ be a Jordan curve in the complex plane $\mC$, and denote by $\Omega$ and $\Omega^*$ the bounded and unbounded complementary components of $\Gamma$ respectively. Suppose for simplicity that $0 \in \Omega$, and let $F:\mD \to \Omega$ and $G:\mD^* \to \Omega^*$ be conformal maps with $F(0)=0$ and $G(\infty)=\infty$. We assume that $F$ and $G$ are normalized appropriately so that they are uniquely determined by $\Gamma$ (the precise choice of normalization will not matter). Note that $F$ and $G$ extend to homeomorphisms on the closure of their respective domain, by Carath\'eodory's theorem on boundary extensions of conformal maps.

\begin{definition}
The \textit{conformal welding} of $\Gamma$ is the orientation-preserving $h_\Gamma: \mT \to \mT$ defined by $h_\Gamma:=G^{-1} \circ F$.
\end{definition}

Denote by $\omega$ and $\omega^*$ the harmonic measure on $\Omega$ and $\Omega^*$ with respect to $0$ and $\infty$ respectively. Recall that harmonic measure for $\mD$ with respect to $0$ is simply normalized Lebesgue measure, thus by conformal invariance $\omega$ and $\omega^*$ are pushforward of normalized Lebesgue measure on $\mT$ under $F:\mT \to \Gamma$ and $G:\mT \to \Gamma$ respectively.

We shall also need the notion of tangent point for a Jordan curve $\Gamma$.

\begin{definition}
Let $\Gamma$ be a Jordan curve in $\mC$ parametrized by $\eta:[0,1] \to \mC$. If $t_0 \in (0,1)$, we say that $\Gamma$ has a \textit{tangent} at $\eta(t_0)$ if there exists an angle $\theta$ such that

\begin{displaymath}
\arg{(\eta(t)-\eta(t_0))} \to \left\{ \begin{array}{ll}
\theta & \textrm{if $t \to t_0^+ $}\\
\theta+\pi & \textrm{if $t \to t_0^-$}.\\
\end{array} \right.
\end{displaymath}
This is independent of the choice of the parametric representation. The set of tangent points of $\Gamma$ is denoted by $T_\Gamma$.
\end{definition}

We will need a result due to Bishop, Carleson, Garnett and Jones relating the behavior of the conformal welding $h_\Gamma$, the size of $T_\Gamma$ as well as the harmonic measures $\omega$ and $\omega^*$.

\begin{theorem}
\label{theorem:singular}
Let $\Gamma$ be a Jordan curve in $\mC$ with conformal welding $h_\Gamma: \mT \to \mT$. Then the following are equivalent:

\begin{enumerate}[\normalfont(i)]
\item $H^1(T_\Gamma)=0$;
\item $\omega \perp \omega^*$;
\item the conformal welding $h_\Gamma: \mT \to \mT$ is singular, i.e. it maps a Borel subset of $\mT$ of full Lebesgue measure onto a set of zero Lebesgue measure.

\end{enumerate}

\end{theorem}

\begin{proof}
The equivalence of (i) and (ii) is a famous theorem of Bishop, Carleson, Garnett and Jones \cite{BishopCarlesonGarnettJones1989}.
\\

If  $\omega \perp \omega^*$, then there is a Borel set $B \subset \Gamma$ with $\omega(B)=1$ but $\omega^*(B)=0$. Let $B':=F^{-1}(B) \subset \mT$. Then $B'$ has full Lebesgue measure in $\mT$ and $h_\Gamma(B')=G^{-1}(B)$ has zero Lebesgue measure. Conversely, if $h_\Gamma$ is singular, then there is a Borel set $B' \subset \mT$ with full Lebesgue measure such that $h_\Gamma(B')$ has zero Lebesgue measure. Letting $B:=F(B')$, we see that $B$ is a Borel subset of $\Gamma$ with $\omega(B)=1$ but $\omega^*(B)=0$, hence $\omega \perp \omega^*$. This shows that (ii) and (iii) are equivalent.
\end{proof}

\section{Matching problem and conformal welding}
\label{sec3}

The proof of Theorem \ref{mainthm} is based on the following characterization of MP-solvable Jordan curves in terms of conformal welding. Recall that $A(\mD)$ is the disk algebra, i.e., the Banach algebra consisting of all continuous functions on $\overline{\mD}$ that are holomorphic on $\mD$.

\begin{theorem}
\label{theorem:welding}
Let $\Gamma$ be a Jordan curve in $\mC$ with conformal welding $h_\Gamma: \mT \to \mT$. Then $\Gamma$ is MP-solvable if and only if there exist two non-constant functions $\varphi, \psi \in A(\mD)$ such that $\varphi = \psi \circ h_\Gamma$ on $\mT$.
\end{theorem}

\begin{proof}
Suppose that there exist two non-constant functions $\varphi, \psi \in A(\mD)$ such that $\varphi = \psi \circ h_\Gamma$ on $\mT$. Subtracting $\psi(0)$ from both $\varphi$ and $\psi$ if necessary, we may assume that $\psi(0)=0$. As before denote by $\Omega$ and $\Omega^*$ the bounded and unbounded complementary components of $\Gamma$ respectively, and let $F:\mD \to \Omega$, $G:\mD^* \to \Omega^*$ be conformal maps with $F(0)=0$ and $G(\infty)=\infty$. Recall that $h_\Gamma:=G^{-1} \circ F$. Define $f:= \varphi \circ F^{-1}$ and
$$g(z):= \overline{\psi \left( \frac{1}{\overline{G^{-1}(z)}}\right)}.$$
Then clearly $f \in A(\Omega)$ and $g \in A(\Omega^*)$ and both functions are non-constant. Note that $g(\infty)=0$ since $G(\infty)=\infty$ and $\psi(0)=0$. Also, if $z \in \Gamma$, then $1/\overline{G^{-1}(z)}=G^{-1}(z)$ and thus
$$\overline{g(z)} = \psi(G^{-1}(z)) = \varphi(F^{-1}(z))=f(z).$$
This shows that $(f,g)$ is a matching pair for $\Gamma$, as required.
\\

Conversely, suppose that $\Gamma$ is MP-solvable, so there exist non-constant functions $f \in A(\Omega), g \in A(\Omega^*)$ such that $f=\overline{g}$ on $\Gamma$. Then we can define two non-constant functions $\varphi, \psi \in A(\mD)$ by
$$
\varphi(z):=f(F(z)), \quad \psi(z):=\overline{g\left(G\left(\frac{1}{\overline{z}}\right)\right)} \qquad (z \in \overline{\mD}).
$$
If $z \in \mT$, then $1/\overline{h_\Gamma(z)}=h_\Gamma(z)$ and we obtain
$$\psi(h_\Gamma(z)) = \overline{g(G(h_\Gamma(z))} = \overline{g(F(z))} = f(F(z))=\varphi(z)$$
as required.
\end{proof}

The matching problem for a given Jordan curve $\Gamma$ is therefore reduced to the following question.

\begin{question}
\label{question:diskalgebra}
Do there exist non-constant functions $\varphi, \psi \in A(\mD)$ such that $\varphi=\psi \circ h_\Gamma$ on $\mT$?
\end{question}

Let us now revisit Example \ref{lemniscate} from this new perspective.

\begin{example}
\label{ex:lemniscates}
Let $p$ be a polynomial of degree $n$ such that $\Gamma:=\{|p(z)|=c\}$ is a Jordan curve for some $c>0$. Replacing $p$ by $p/c$ if necessary, we may assume that $c=1$. As before denote by $\Omega$ and $\Omega^*$ the bounded and unbounded complementary components of $\Gamma$ respectively, and let $F:\mD \to \Omega$ and $G:\mD^* \to \Omega^*$ be conformal maps with $F(0)=0$ and $G(\infty)=\infty$. Then $p \circ F$ is a degree $n$ proper holomorphic map of the unit disk $\mD$ onto itself, in other words a Blaschke product $B$. Similarly, the map $p \circ G$ is a degree $n$ proper holomorphic map of $\mD^*$ onto itself, with only one pole, at $\infty$. It follows that there exists $a \in \mT$ such that $p(G(z))=az^n$. Now recall from the definition of conformal welding that $G \circ h_\Gamma=F$ on $\mT$. Composing both sides by $p$ gives $\varphi = \psi \circ h_\Gamma$ on $\mT$, where $\varphi(z):=B(z)$ and $\psi(z):=az^n$. This was first observed by Ebenfelt, Khavinson and Shapiro in \cite[Theorem 2.2]{EbenfeltKhavinsonShapiro2011}. It follows from Theorem \ref{theorem:welding} that $\Gamma$ is MP-solvable, a fact consistent with Example \ref{lemniscate}.

More generally, suppose as in Example \ref{lemniscate} that $r$ is a rational map of degree $n$ such that the lemniscate $\Gamma:=\{z \in \hC:|r(z)|=1\}$ is a Jordan curve. Assume in addition that
\begin{enumerate}[(i)]
\item $r$ has no pole in $\Omega$;
\item $r$ has no zero in $\Omega^*$;
\item $r(\infty)=\infty.$
\end{enumerate}
Then just as in the polynomial case, we obtain two Blaschke products $A$ and $B$ of degree $n$ such that $A \circ h_\Gamma=B$. Taking $\varphi:=B$ and $\psi:=A$ gives $\varphi = \psi \circ h_\Gamma$ on $\mT$, as observed in \cite[Theorem 3.2]{Younsi2016}. It follows from Theorem \ref{theorem:welding} that $\Gamma$ is MP-solvable, which is again consistent with Example \ref{lemniscate}.
\end{example}

We shall now prove Theorem \ref{mainthm} by constructing new examples of curves for which the answer to Question \ref{question:diskalgebra} is positive. The proof relies on the Bishop--Carleson--Garnett--Jones theorem (Theorem \ref{theorem:singular}) as well as on the following theorem of Browder and Wermer from \cite[Theorem 2]{BrowderWermer1964}.

\begin{theorem}[Browder--Wermer]
\label{theorem:browder}
Let $h:\mT \to \mT$ be an orientation-preserving homeomorphism. If $h$ is singular in the sense of (iii) in Theorem \ref{theorem:singular}, then
$$A_h:= \{\psi \in A(\mD): \, \mbox{there exists} \, \, \varphi \in A(\mD) \, \, \mbox{such that}\, \, \varphi=\psi \circ h \, \, \mbox{on} \, \, \mT\}$$
is a \textit{Dirichlet algebra} on $\mT$, meaning that the real parts of functions in $A_h$ are uniformly dense in the space of real-valued continuous functions on $\mT$.
\end{theorem}
See also the work of Bishop in \cite[Corollary 6.1]{Bishop1987} for a constructive proof.

The conclusion of Theorem \ref{theorem:browder} is quite strong but we shall only need the fact that if $h$ is singular then $A_h$ contains non-constant functions. We can now prove Theorem \ref{mainthm}.
\begin{proof}
Let $\Gamma$ be a Jordan curve in $\mC$ with $H^1(T_\Gamma)=0$, where $T_\Gamma$ denotes the set of tangent points of $\Gamma$. Denote by $h_\Gamma: \mT \to \mT$ the conformal welding of $\Gamma$. By Theorem \ref{theorem:singular}, the homeomorphism $h_\Gamma$ is singular. In particular, by Theorem \ref{theorem:browder}, there exist non-constant functions $\varphi, \psi \in A(\mD)$ such that $\varphi=\psi \circ h_\Gamma$ on $\mT$. By Theorem \ref{theorem:welding}, the curve $\Gamma$ is MP-solvable, as required.
\end{proof}

\section{MP-solvable curves of arbitrary dimension}
\label{sec4}

In this section we prove Corollary \ref{mainc} which says that for any $\alpha \in [1,2]$, there exists an MP-solvable Jordan curve $\Gamma$ with $\dim_H(\Gamma)=\alpha$.
\begin{proof}
In view of Theorem \ref{mainthm}, it suffices to construct Jordan curves of any prescribed dimension in $[1,2]$ and without tangent points. Examples of such curves were constructed by Bishop in his Ph.D. thesis, see \cite[Chapter II, Section 6]{Bishop1987} (see also \cite[Theorem 1.1]{Bishop2002} for a curve of dimension $1$). Moreover, a construction due to Osgood \cite{Osgood1903} gives a Jordan curve with positive area and no tangent points. 

We also mention that for $\alpha \in (1,2)$, one can construct a Jordan curve $\Gamma$ with $\dim_H(\Gamma)=\alpha$ without tangent points using iterated function systems, more precisely, generalized Koch curves. Following \cite[Chapter 7, Section 2.1]{Stein2005}, let $l \in (0.25, 0.5)$. First, define the following four affine transformations
$$
            f_1 (z) = lz, \, f_2 (z) = l e^{i \theta} z + l , \, f_3 (z) = l e^{-i\theta} z + \tfrac{1}{2} + b, \, f_4 (z) = l z + 1 - l ,
$$
where $b := \sqrt{l - \tfrac{1}{4}}$ and $\theta = \arctan (b / (1/2 - l))$. These are similitudes of $\mC$ and each map has a contraction ratio $r_j$ equals to $l$ ($1 \leq j \leq 4$). The unique compact set $K_l$ satisfying
$$
            K_l = f_1 (K_l) \cup f_2 (K_l) \cup f_3 (K_l) \cup f_4 (K_l)
$$
is the attractor of the iterated function system $\{ f_1, f_2, f_3, f_4 \}$. The set $K_l$ is usually called a \textit{generalized Koch curve}. Approximations of these curves are presented in Figure \ref{fig:KochCurves}, where $K_l^0 := [0, 1]$ and
$$
            K_l^n := f_1 (K_l^{n-1}) \cup f_2 (K_l^{n-1}) \cup f_3 (K_l^{n-1}) \cup f_4 (K_l^{n-1}) \quad (n \geq 1 ) .
$$
\begin{figure}[H]
        \centering
        \begin{subfigure}{0.46\textwidth}
        \centering
        \includegraphics[scale=0.45]{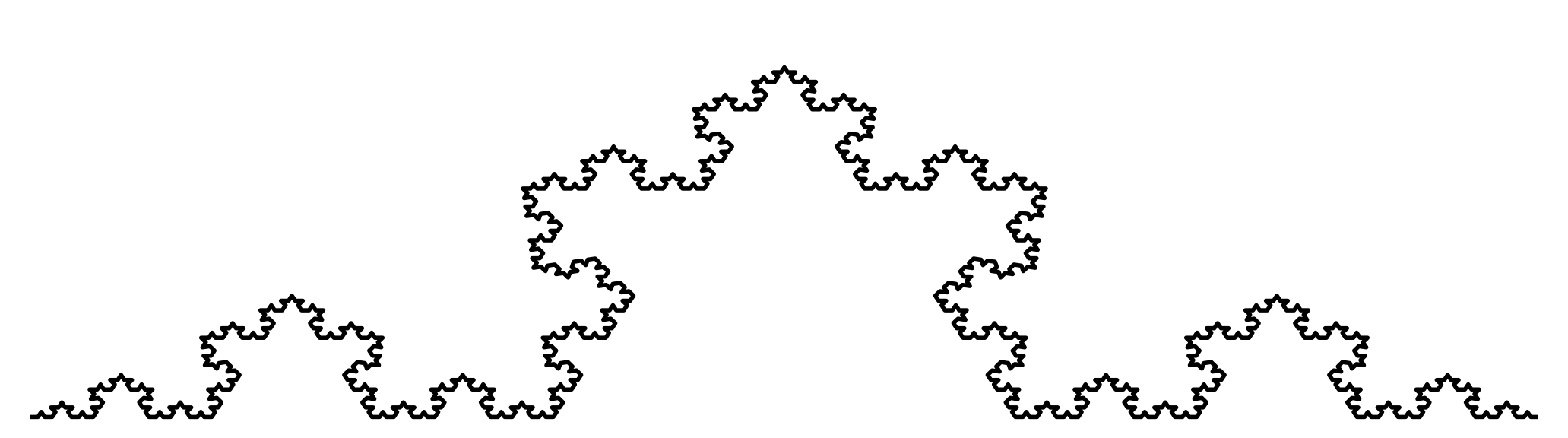}
        \caption{$l = 0.3468$}
        \end{subfigure}
        \begin{subfigure}{0.45\textwidth}
        \centering
        \includegraphics[scale=0.46]{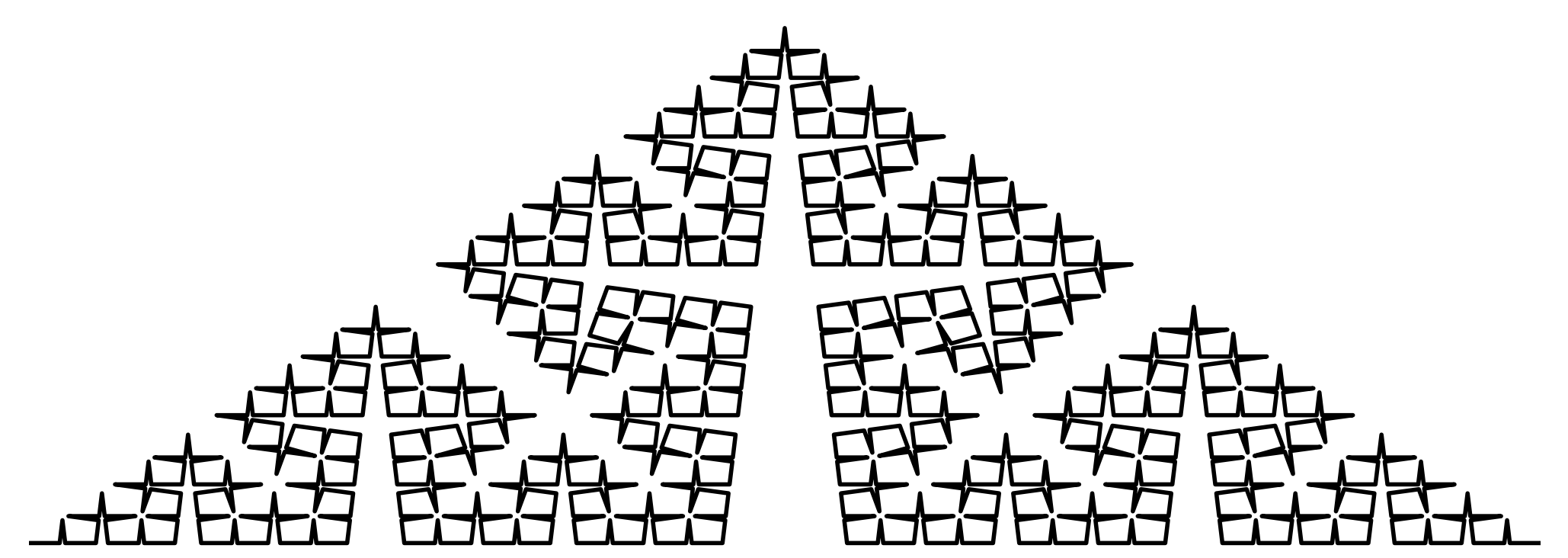}
        \caption{$l = 0.4588$}
        \end{subfigure}
        \caption{The approximation $K_l^{5}$ for some values of $l$}\label{fig:KochCurves}
\end{figure}

One can check that the maps in the iterated function system satisfy the \textit{open set condition} (see \cite[Definition 2.2.1]{BishopPeres2017}): there exists a bounded open set $V \subset \mC$ such that $f_j (V) \subset V$ for $1 \leq j \leq 4$ and $f_j (V) \cap f_k (V) = \emptyset$ for $j \neq k$. An example of such $V$ is the interior of the isosceles triangle with vertices at $0$, $1/2 + bi$, and $1$, and angles $\theta / 2$ and $\pi - \theta$. It follows from that condition that the curve is not self-intersecting and by \cite[Theorem 2.12]{Stein2005}, its Hausdorff dimension $s$ is

$$
            s = \frac{\log 4}{\log (1 / l)} .
$$
As $l \rightarrow 1/4$, the number $s$ approaches $1$ and as $l \rightarrow 1/2$, it approaches $2$. Also, the set of tangent points of $K_l$ is empty.

Now define $\Gamma_l$ as followed:
$$
            \Gamma_l := K_l \cup e^{-i\pi / 3} (K_l') \cup (e^{-2\pi / 3} K_l + 1) ,
$$
where $K_l':= \{ \overline{z} \, : \, z \in K_l \}$. The curve $\Gamma_l$ is called a generalized Koch snowflake and it is a Jordan curve for $1/4 < l < 1/2$. See Figure \ref{fig:KochSnowflakes}.
\begin{figure}[H]
        \centering
        \begin{subfigure}{0.45\textwidth}
        \centering
        \includegraphics[scale=0.45]{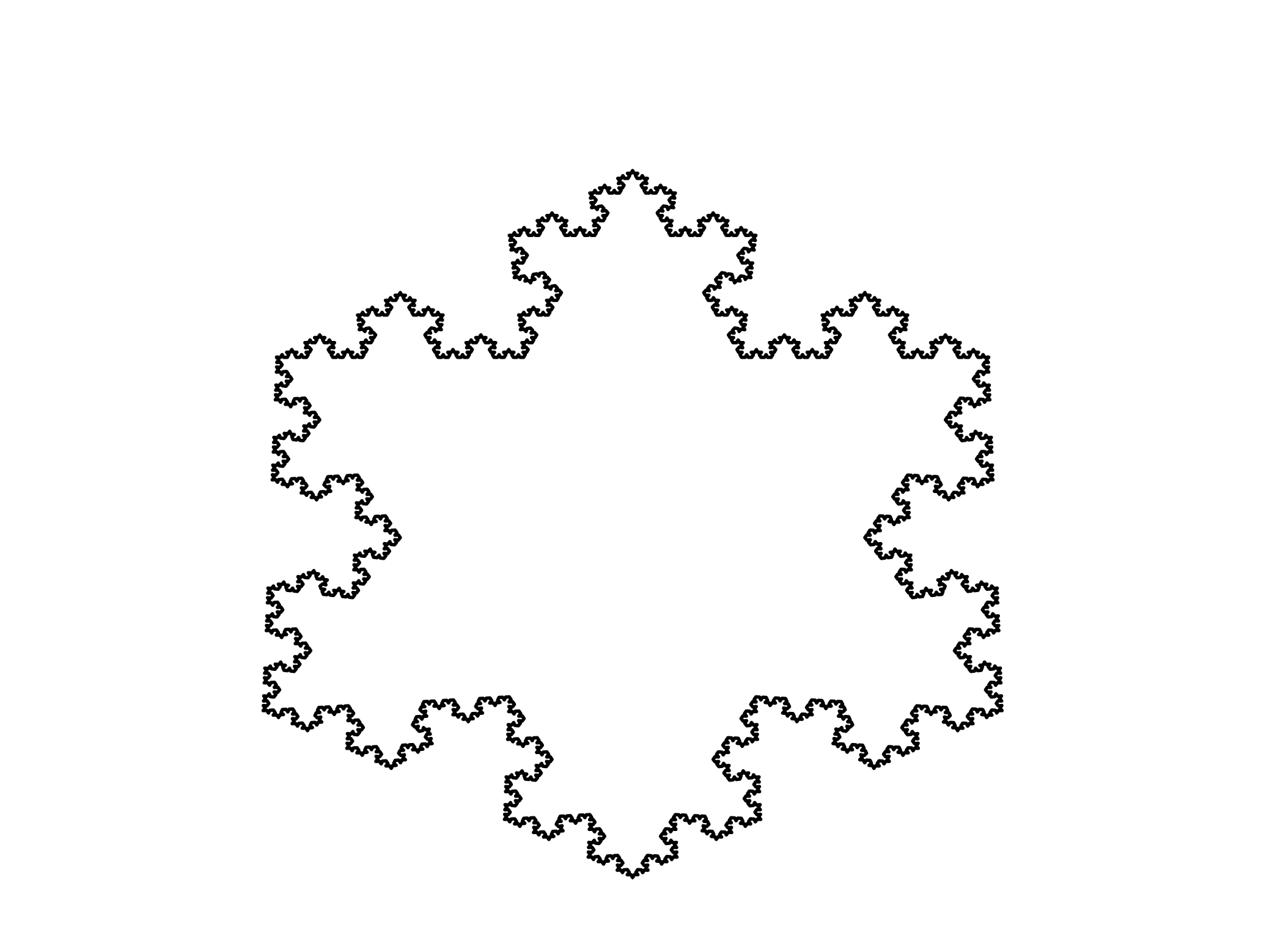}
        \caption{$l = 0.3468$}
        \end{subfigure}
        \begin{subfigure}{0.45\textwidth}
        \centering
        \includegraphics[scale=0.45]{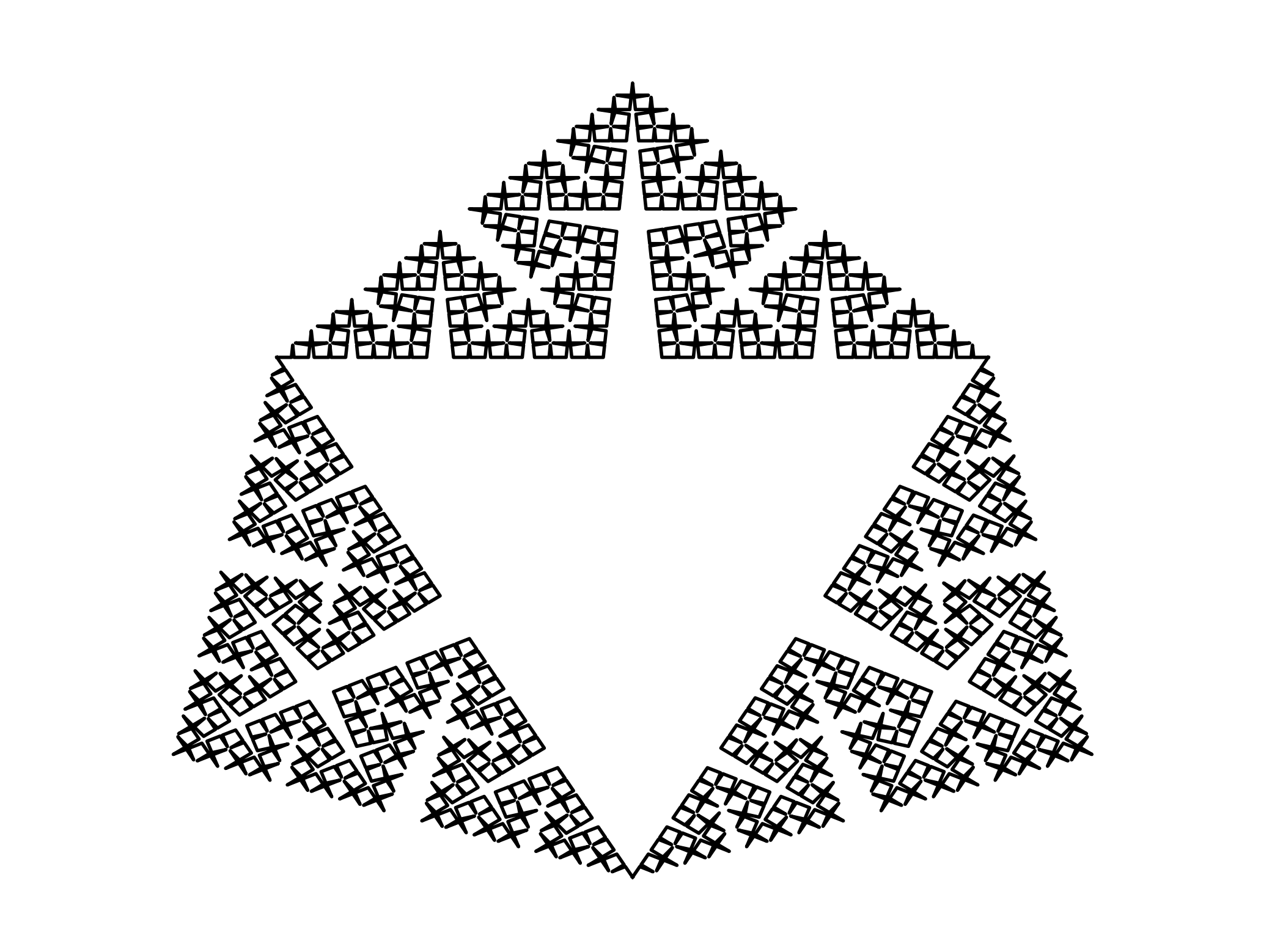}
        \caption{$l = 0.4588$}
        \end{subfigure}
        \caption{Approximation of $\Gamma_l$ for some values of $l$}\label{fig:KochSnowflakes}
\end{figure}
The Hausdorff dimension of the Jordan curve $\Gamma_l$ also equals $s = \log(4) / \log(1 / l) \in (1, 2)$ and it has no tangent point, as desired.
\end{proof}

\section{Curves for which the matching problem has no solution}
\label{sec5}

In this section we prove Theorem \ref{mainthm2}.

\begin{proof}
Let $\Gamma$ be a Jordan curve in $\mC$. Suppose that there exists an open half-plane $\mathbb{H}$ such that $\Gamma \subset \overline{\mathbb{H}}$ and $\Gamma \cap \partial \mathbb{H}$ is a line segment. We have to show that $\Gamma$ is not MP-solvable.

As before denote by $\Omega$ and $\Omega^*$ the bounded and unbounded complementary components of $\Gamma$ respectively. Suppose for a contradiction that there exist non-constant functions $f \in A(\Omega)$ and $g \in A(\Omega^*)$ such that $f=\overline{g}$ on $\Gamma$.

Rotating and translating if necessary, we may assume that $\mathbb{H}$ is the upper half-plane $\{z \in \mC: \Im{z}>0\}$, so that $\Gamma \cap \partial \mathbb{H} = \Gamma \cap \mR$ is a line segment in the real line. Denote by $\Gamma'$ the image of $\Gamma$ under $z \mapsto \overline{z}$. Also, let $I$ be the interval $\Gamma \cap \mR$ with its endpoints removed. Then $J:=(\Gamma \cup \Gamma') \setminus I$ is a Jordan curve (see Figure \ref{fig:illustrationProof}); denote by $D$ its unbounded complementary component. Note that $D$ is contained in $\Omega^*$ and is preserved by complex conjugation $z \mapsto \overline{z}$.
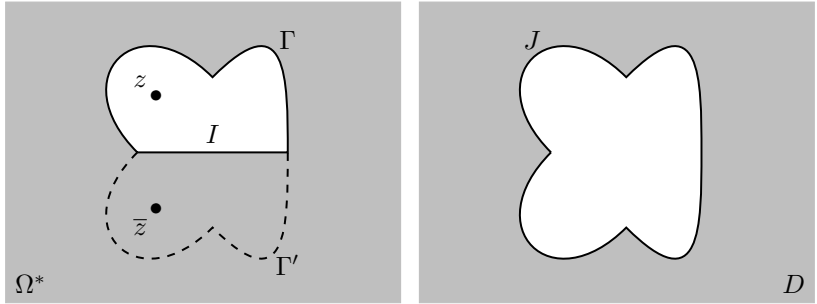
\begin{figure}[ht]
    \centering
    \begin{tikzpicture}
       \begin{scope}[shift={(-2.75cm, 0)}]
            \draw[black!25, fill=black!25] (-2.75, -2) -- (2.5, -2) -- (2.5, 2) -- (-2.75, 2) -- cycle;
            \draw[black, line width=0.75pt, dashed] (-1, 0) .. controls (-2, -1) and (-1, -2) .. (0, -1) .. controls (1, -2) and (1, -1) .. (1, 0);
            \draw[black, line width=0.75pt, fill = white] (-1, 0) .. controls (-2, 1) and (-1, 2) .. (0, 1) .. controls (1, 2) and (1, 1) .. (1, 0) -- (-1, 0)node[midway, above]{$I$};
            \draw[black] (1, 1.5)node{$\Gamma$};
            \draw[black, line width=0.75pt] (1, -1.5)node{$\Gamma'$};
            \draw[black] (-0.75, 0.75)node{\textbullet};
            \draw[black] (-0.75, 0.75)node[above left]{$z$};
            \draw[black] (-0.75, -0.75)node{\textbullet};
            \draw[black, dashed] (-0.75, -0.75)node[below left]{$\overline{z}$};
            \draw[black] (-2.75, -2)node[above right]{$\Omega^*$};
       \end{scope}
       \begin{scope}[shift={(2.75cm, 0)}]
            \draw[black!25, fill=black!25] (-2.75, -2) -- (2.5, -2) -- (2.5, 2) -- (-2.75, 2) -- cycle;
            \draw[black, line width=0.75pt, fill = white] (-1, 0) .. controls (-2, 1) and (-1, 2) .. (0, 1) .. controls (1, 2) and (1, 1) .. (1, 0) .. controls (1, -1) and (1, -2) .. (0, -1) .. controls (-1, -2) and (-2, -1) .. (-1, 0);
            \draw[black] (-1.25, 1.5)node{$J$};
            \draw[black] (2.5, -2) node[above left]{$D$};
       \end{scope}
    \end{tikzpicture}
    \caption{The Jordan curve $\Gamma$, its reflection $\Gamma'$, and the curve $J$}
    \label{fig:illustrationProof}
\end{figure}
Now, since $g$ is analytic in the lower half-plane, we have that the function $h(z):=\overline{g(\overline{z})}$ is analytic in $\mathbb{H}$ and in particular in $\Omega$. It also extends continuously to $\Gamma$. Moreover, for $z \in I$ we have $f(z)=\overline{g(z)} = \overline{g(\overline{z})} = h(z).$ Precomposing both $f$ and $h$ by a conformal map from $\mD$ onto $\Omega$, we obtain two functions in the disk algebra which coincide on a subarc of $\mT$, hence must coincide everywhere in the whole disk. It follows that $f(z)=h(z)$ for all $z \in \overline{\Omega}$, and in particular for all $z \in \Gamma$.

Now, note that the functions $g$ and $h$ are both analytic in $D$ and continuous up to the boundary curve $J$. We claim that $g=\overline{h}$ on $J$. Indeed, if $z \in J$ then either $z \in \Gamma$ or $z \in \Gamma'$. In the first case we get $g(z)=\overline{f(z)}=\overline{h(z)}$ as required, since $f=h$ on $\Gamma$. In the second case we have $\overline{z} \in \Gamma$, thus
$$
\overline{h(z)} = g(\overline{z}) = \overline{f(\overline{z})} = \overline{h(\overline{z})}=g(z),
$$
again as required.

To summarize, we have that $g$ is analytic in the Jordan domain $D$ and continuous up to the boundary, that $\overline{h}$ is anti-analytic in $D$ and also continuous up to the boundary, with $g=\overline{h}$ on $\partial D$. It follows that $g$ and $\overline{h}$ must be constant, as can be seen by precomposing both functions by a conformal map from $\mD$ onto $D$ and comparing Fourier coefficients. This contradicts the fact that $(f,g)$ is a matching pair for $\Gamma$. It follows that $\Gamma$ is not MP-solvable, as required.
\end{proof}

\begin{corollary}
\label{corollary:dense}
For any Jordan curve $\Gamma$ and every $\epsilon>0$, there exists an MP-solvable Jordan curve $\Gamma'$ and a non MP-solvable Jordan curve $\Gamma''$ such that
$$\operatorname{dist}_H(\Gamma,\Gamma')<\epsilon, \, \operatorname{dist}_H(\Gamma,\Gamma'')<\epsilon. $$
Here $\operatorname{dist}_H$ denotes the Hausdorff distance.
\end{corollary}

\begin{proof}
The existence of $\Gamma'$ follows directly from Example \ref{lemniscate} combined with Hilbert's lemniscate theorem on the density of polynomial lemniscates in the set of all Jordan curves with respect to the Hausdorff distance. Also, it is easy to see that the Jordan curves as in Theorem \ref{mainthm2} are dense, which gives the existence of $\Gamma''$. Alternatively the existence of $\Gamma''$ can also be deduced from Theorem \ref{nomatching}: first approximate $\Gamma$ by an analytic Jordan curve $\tilde{\Gamma}$, then truncate the power series expansion of the Riemann map from $\mD$ onto the bounded complementary component of $\tilde{\Gamma}$. The image of $\mT$ under that truncated power series is not MP-solvable by Theorem \ref{nomatching} and arbitrarily close to $\Gamma$ in the Hausdorff distance.
\end{proof}

\section{Proof of Theorem \ref{answerq1}}
\label{sec6}

This section contains the proof of Theorem \ref{answerq1}, which says that a set $G \subset \hC$ is a rational lemniscate if and only if the three conditions (\ref{top_cond}), (\ref{analyt_cond}) and (\ref{alg_cond}) are satisfied.

\subsection{The three conditions are necessary}

Suppose $G=\{z \in \hC: |r(z)|=1\}$ is a rational lemniscate. We shall prove separately that $G$ satisfies each of the conditions (\ref{top_cond}), (\ref{analyt_cond}) and (\ref{alg_cond}).

\begin{lemma}
Condition \emph{(\ref{top_cond})} is satisfied.
\end{lemma}

\begin{proof}
We have to show that $G$ is a lemniscate graph. A complete proof of this fact can be found in \cite[Proposition 2.4]{BishopEremenkoLazebnik2025}, but we describe the main ideas here for the reader's convenience. Note that $G$ is locally homeomorphic to a line segment away from the critical points of $r$, and thus $G$ with the critical points of $r$ removed is a $1$-manifold, which therefore must consist of components each of which is either an (open) Jordan arc or a (closed) Jordan curve. See e.g. \cite{Milnor1997}. On the other hand, the normal form for rational maps at a critical point imply that there are $2(d+1)\geq4$ edges meeting at any degree $d$ critical point of $r$ lying on $G$. It follows that $G$ is a lemniscate graph as required.
\end{proof}

\begin{lemma}
\label{lemma:analyt}
Condition \emph{(\ref{analyt_cond})} is satisfied.
\end{lemma}

\begin{proof}
We have to find points $(z_i)_{i=1}^m \subset \hC \setminus G$, not necessarily distinct, so that for any two faces $E$, $F$, we have
$$
\sum_{z_i \in E} \omega(B, z_i, E) = \sum_{z_i \in F} \omega(B, z_i, F)
$$
for every Borel set $B\subset \partial E\cap \partial F$.

Recall that $G=\{z \in \hC: |r(z)|=1\}$ for some rational map $r$. In this case the zeros and poles are the obvious candidates for the points $z_i$; set $(z_i)_{i=1}^m:=r^{-1}(\{0,\infty\}) \subset \hC \setminus G$, where we repeat a zero or pole of multiplicity $d$ as many times in the list $(z_i)_{i=1}^m$.

Let $E$ and $F$ be two faces. We may suppose that their boundaries intersect, otherwise there is nothing to prove. So suppose $E$ and $F$ are adjacent, and assume without loss of generality that $E$ is white and $F$ is grey. Since $r:F \to \mD^*$ is a proper holomorphic mapping with poles exactly at the points in $(z_i)_{i=1}^m$ which belong to $F$, we deduce from Proposition \ref{proper1} that
$$
\log{|r(z)|} = \sum_{z_i \in F} g_F(z,z_i) \qquad (z \in F).
$$

Similarly, using the fact that $r:E \to \mD$ is a proper holomorphic mapping, we get
$$
\log{|r(z)|} = -\sum_{z_i \in E} g_E(z,z_i) \qquad (z \in E).
$$

For $\zeta \in \partial E \cap \partial F$ not a vertex, this gives
$$-\frac{1}{2\pi} \sum_{z_i \in E} \frac{\partial g_E (\zeta , z_i)}{\partial n_\zeta} |d\zeta|=\frac{1}{2\pi} \frac{\partial \log{|r(\zeta)|}}{\partial n_\zeta} |d\zeta| = \frac{1}{2\pi} \sum_{z_i \in F} \frac{\partial g_F (\zeta , z_i)}{\partial n_\zeta} |d\zeta|,$$
where $n_\zeta$ is the unit normal at $\zeta$, outer with respect to $E$. But by Proposition \ref{extension}, the above becomes
$$
\sum_{z_i \in E} d\omega(\zeta,z_i,E) = \sum_{z_i \in F} d\omega(\zeta,z_i,F)
$$
as measures on $(\partial E \cap \partial F) \setminus V$, where $V \subset G$ is the set of vertices. But $V$ is a finite set and therefore has zero harmonic measure, hence
$$
\sum_{z_i \in E} \omega(B, z_i, E) = \sum_{z_i \in F} \omega(B, z_i, F)
$$
for every Borel set $B \subset \partial E \cap \partial F$, as required.
\end{proof}

\begin{lemma}
Condition \emph{(\ref{alg_cond})} is satisfied.
\end{lemma}

\begin{proof}
Let $(z_i)_{i=1}^m \subset \hC \setminus G$ be as in the proof of Lemma \ref{lemma:analyt}. We have to show that for any face $E$ we have
$$\sum_{z_i \in E} \omega(\Gamma, z_i, E) \in \mathbb{Z}$$
for every component $\Gamma$ of $\partial E$. But this follows immediately from Proposition \ref{proper2} and the fact that $r$ is a proper holomorphic mapping of $E$ onto either $\mD$ or $\mD^*$.
\end{proof}

\subsection{The three conditions are sufficient}

In this subsection we prove that the three conditions (\ref{top_cond}), (\ref{analyt_cond}) and (\ref{alg_cond}) in Theorem \ref{answerq1} are sufficient for a set $G \subset \hC$ to be a rational lemniscate. The main difficulty is showing that the three conditions imply that all the edges of $G$ must be analytic (recall that we do not assume any regularity for $G$ in Theorem \ref{answerq1}). In order to overcome this, we will need the following result of Dubinin \cite[Theorem 2]{Dubinin2008}

\begin{theorem}
\label{dubinin}
Let $D$ and $D'$ be domains in $\hC$ bounded by finitely many Jordan curves, and let $w_0 \in G$. Let $B:D \to D'$ be a proper holomorphic mapping, so that $B$ extends continuously to $\overline{D}$ with $B(\partial D) = \partial D'$. Let $\gamma \subset \partial D$ be a Jordan arc on which $B$ is injective. Then
$$\omega(B(\gamma),w_0,D') = \sum_{k=1}^n \omega(\gamma,z_k,D),$$
where $z_1, \dots, z_n$ are the zeros of $B-w_0$ if $w_0 \neq \infty$ or the poles of $B$ if $w_0 = \infty$, repeated according to multiplicities.
\end{theorem}

We shall in fact need a version of Theorem \ref{dubinin} for proper holomorphic mappings of lemniscate domains onto either $\mD$ or $\mD^*$.

\begin{corollary}
\label{corollarylength}
Let $\Omega \subset \hC$ be a lemniscate domain and let $V \subset \partial \Omega$ be the set of vertices. Let $B:\Omega \to D'$ be a proper holomorphic mapping, where $D'$ is either $\mD$ or $\mD^*$. Let $z_1, \dots, z_n$ be the zeros of $B$ (in the case $D'=\mD$) or the poles of $B$ (in the case $D'=\mD^*$). Then
$$
\operatorname{length}(B(\gamma))=\sum_{k=1}^n \omega(\gamma,z_k,\Omega)
$$
for each Jordan arc $\gamma \subset \partial \Omega \setminus V$ on which $B$ is injective.
\end{corollary}

\begin{proof}
We only prove the result in the case $D'=\mD$, the other case $D'=\mD^*$ follows by composing with $z \mapsto 1/z$. So let $B:\Omega \to \mD$ be a proper holomorphic mapping with zeros $z_1, \dots, z_n$. Recall that the map $B$ extends continuously to $\partial \Omega \setminus V$, by Proposition \ref{properextension}. Let $\gamma \subset \partial \Omega \setminus V$ be a Jordan arc on which $B$ is injective. Then $B(\gamma)$ is an arc in $\mT$ hence $\operatorname{length}(B(\gamma))$ is well-defined.

First note that if $\Omega$ is an analytic Jordan domain, then the result follows immediately from Theorem \ref{dubinin}, since
$$\omega(B(\gamma),0,\mD) =  \operatorname{length}(B(\gamma)).$$

For the general case where $\Omega$ is a lemniscate domain, let $f:\Omega \to \Omega'$ be a conformal map where $\Omega'$ is an analytic Jordan domain, as in the remark following Definition \ref{lemniscatedomain}. Then $f$ has a continuous and injective extension to $\overline{\Omega} \setminus V$. Let $B:\Omega \to \mD$ be a proper holomorphic mapping with zeros $z_1, \dots, z_n$, and let $\gamma \subset \partial \Omega \setminus V$ be a Jordan arc on which $B$ is injective. Then $B \circ f^{-1} : \Omega' \to \mD$ is a proper holomorphic mapping with zeros $f(z_1),\dots,f(z_n)$, and it is injective on the Jordan arc $\gamma':=f(\gamma) \subset \partial \Omega'$. Thus
$$\omega(B \circ f^{-1}(\gamma'),0,\mD) = \sum_{k=1}^n \omega(\gamma',f(z_k),\Omega')$$
by the analytic case. But the left-hand side is just $\operatorname{length}(B(\gamma))$, and conformal invariance of harmonic measure gives
$$\operatorname{length}(B(\gamma)) = \sum_{k=1}^n \omega(\gamma,z_k,\Omega)$$
as required.
\end{proof}

We can now prove that the three conditions (\ref{top_cond}), (\ref{analyt_cond}) and (\ref{alg_cond}) in Theorem \ref{answerq1} are sufficient for a set $G \subset \hC$ to be a rational lemniscate.

\begin{proof}
Let $G$ be a subset of $\hC$ satisfying the three conditions (\ref{top_cond}), (\ref{analyt_cond}) and (\ref{alg_cond}). Then $G$ is a lemniscate graph and there exist points $(z_i)_{i=1}^m \subset \hC \setminus G$, not necessarily distinct, so that for any two faces $E$, $F$, we have the analytic condition
$$
\sum_{z_i \in E} \omega(B, z_i, E) = \sum_{z_i \in F} \omega(B, z_i, F)
$$
for every Borel set $B\subset \partial E\cap \partial F$, and the algebraic condition
$$
\sum_{z_i \in E} \omega(\Gamma, z_i, E) \in \mathbb{Z}
$$
for every component $\Gamma$ of $\partial E$.

First note that if $E$ is any white face, then there exists a proper holomorphic mapping of $E$ onto $\mD$ with zeros exactly at the points $z_i \in E$, by the algebraic condition and Proposition \ref{proper2}. Similarly, if $F$ is any grey face, then there exists a proper holomorphic mapping of $F$ onto $\mD^*$ with poles exactly at the points $z_i \in F$. It is easy to see that we can compose all of these proper holomorphic mappings by suitable rotations in order to make sure that for any two adjacent faces $E$ and $F$, the corresponding proper holomorphic mappings agree at a point of the common edge. All these proper holomorphic mappings define a map
$$r: \hC \setminus G \to \mD \sqcup \mD^*.$$
We claim that this map extends continuously to $G \setminus V$. Indeed, let $E$ and $F$ be two adjacent faces, with $E$ white and $F$ grey. Denote by $\gamma$ their common edge, so that $\gamma$ is either a Jordan curve or an open simple arc satisfying $\overline{\gamma} \setminus \gamma \subset V$. Let $B_E:E \to \mD$ and $B_F:F \to \mD^*$ the proper holomorphic mappings previously constructed corresponding to $E$ and $F$ respectively. Then by Proposition \ref{properextension}, both maps extend continuously to $\gamma$, and they agree at some point $\zeta_0 \in \gamma$ by construction. Moreover, both restrictions $B_E:\gamma \to B_E(\gamma)$ and $B_F:\gamma \to B_F(\gamma)$ are finite degree covering maps. It follows that we can partition $\gamma$ into consecutive subarcs $\gamma_1, \dots, \gamma_N$ such that both maps $B_E$ and $B_F$ are injective on each $\gamma_j$, $j \in \{1, \dots, N\}$. We may also assume that the subarc $\gamma_1$ starts at $\zeta_0 $. Let $\zeta$ be any point in the subarc $\gamma_1$, and denote by $\gamma_\zeta$ the subarc of $\gamma_1$ starting at $\zeta_0 $ and ending at $\zeta$. By Corollary \ref{corollarylength}, we have
$$
\operatorname{length}(B_E(\gamma_\zeta))=\sum_{z_i \in E} \omega(\gamma_\zeta, z_i, E)
$$
and
$$
\operatorname{length}(B_F(\gamma_\zeta))=\sum_{z_i \in F} \omega(\gamma_\zeta, z_i, F).
$$
It then follows from the analytic condition (\ref{analyt_cond}) that $B_E(\gamma_\zeta)$ and $B_F(\gamma_\zeta)$ are two subarcs of $\mT$ of equal length, both starting at the point $B_E(\zeta_0)=B_F(\zeta_0)$. In particular they must end at the same point, i.e. $B_E(\zeta)=B_F(\zeta)$. This holds for all $\zeta$ in the subarc $\gamma_1$, in particular for its end point $\zeta_1$. We can now repeat this argument with $\gamma_1$ replaced by $\gamma_2$ to obtain $B_E(\zeta)=B_F(\zeta)$ for all $\zeta \in \gamma_2$, and by induction $B_E(\zeta)=B_F(\zeta)$ for all $\zeta \in \bigcup_{j=1}^N \gamma_j = \gamma$. This holds for all pairs of adjacent faces $E$ and $F$, so that the map $r$ extends continuously to $G \setminus V$ as claimed.

Next we note that each edge of $G$ must be analytic. Indeed, by considering a local inverse of $r$, we see that any non-vertex point on $G$ has a neighborhood $U$ so that $U \cap G$ is an analytic image of a subarc of $\mT$. By standard removability results e.g. \cite[Theorem 10.48]{Zakeri2021}, the map $r$ extends to a meromorphic function on $\hC \setminus V$. It is also bounded near each vertex by construction, so is in fact meromorphic on the whole Riemann sphere $\hC$. This shows that $r:\hC \to \hC$ is a rational map, and by construction
$$L_r = \{z \in \hC: |r(z)|=1\} = r^{-1}(\mT) = G$$
as required.
\end{proof}

\section{Conclusion and open problems}
\label{sec7}

The examples of MP-solvable Jordan curves in Theorem \ref{mainthm} are fractal and their set of tangent points cannot have positive $\mathcal{H}^1$ measure. This raises the following question.

\begin{question}
Let $\Gamma$ be a Jordan curve in $\mC$. If $\Gamma$ is MP-solvable and sufficiently regular (say analytic), must $\Gamma$ be a rational lemniscate?
\end{question}

In view of Theorem \ref{theorem:welding} and Example \ref{ex:lemniscates}, constructing an MP-solvable analytic Jordan curve which is not a rational lemniscate amounts to finding an analytic diffeomorphism $h:\mT \to \mT$ for which the functional equation on $\mT$
$$
\varphi = \psi \circ h
$$
admits non-constant solutions $\varphi,\psi$ in the disk Algebra $A(\mD)$, but no Blaschke product solutions. Since the space $A(\mD)$ is much larger than the collection of all Blaschke products, it seems possible that such $h$ exists, although we were unable to prove it.
\\

We conclude this section by discussing another open question. In Theorem \ref{answerq1} we answered the question of which subsets of $\hC$ are rational lemniscates. The following question is also relevant for the matching problem.

\begin{question}
\label{q2}
Given a rational map $r$, are there necessary and sufficient analytic conditions for the corresponding lemniscate $L_r$ to be a Jordan curve?
\end{question}

A necessary condition for $L_r$ to be a Jordan curve is that all the critical values of $r$ belong to $\mD$. For polynomials, this condition is both necessary and sufficient, but not for rational maps in general. For this reason Question \ref{q2} turns out to be quite subtle and we were only able to obtain the following partial answer.

\begin{theorem}
\label{answerq2}
Let $r$ be a rational map of degree $k \geq 1$ with $|r(\infty)|>1$. Then the rational lemniscate $L_r:=\{z \in \hC:|r(z)|=1\}$ is an analytic Jordan curve if and only if
\begin{enumerate}[(i)]
\item $r$ has precisely $k-1$ critical values in $\mD$ and $k-1$ critical values in $\mD^*$ counting multiplicities; and
\item there exists a Jordan curve $\Gamma$ in $\mC$ whose interior contains all the zeros of $r$ and whose exterior contains all the poles of $r$ such that $|r(z)| \leq 1$ for all $z \in \Gamma$.
\end{enumerate}
\end{theorem}

For the proof of Theorem \ref{answerq2}, we need some lemmas. For a given rational map $r$ with $|r(\infty)|>1$, let $\Omega_r:=\{z \in \hC: |r(z)|<1\}$, so that $\Omega_r$ is a bounded subset of $\mC$ and $\partial \Omega_r = L_r$ is compact. Note that $\Omega_r$ need not be connected in general.

\begin{lemma}
\label{lemma:simplyconnected}
Let $r,L_r,\Omega_r$ be as above. If $\Omega_r$ is simply connected and if $r$ has no critical point on $L_r$, then $L_r$ is an analytic Jordan curve.
\end{lemma}

\begin{proof}
It is well-known that an open subset of the plane is simply connected if and only if its boundary is connected. Hence $L_r$ is connected. Since the restriction $r:L_r \to \mT$ is non-singular, it follows that $L_r$ is an analytic Jordan curve, as required.
\end{proof}

\begin{lemma}
\label{lemma:zeros}
Each connected component of $\Omega_r$ contains at least one zero of $r$.
\end{lemma}

\begin{proof}
Let $U$ be a component of $\Omega_r$. Then $U$ is bounded and contains no pole of $r$. Moreover, we have that $|r|=1$ on $\partial U$, since $\partial U \subset L_r$. If $U$ does not contain any zero of $r$, then we can apply the maximum modulus principle to both $r$ and $1/r$ to deduce that $|r| \equiv 1$ in $U$, contradicting the open mapping theorem. Therefore $U$ must contain at least one zero of $r$, as required.
\end{proof}

\begin{lemma}[Riemann--Hurwitz formula]
\label{lemma:rh}
Let $U,V \subset \hC$ be finitely connected domains and let $f:U \to V$ be a degree $k$ branched covering with $m$ branch points counting multiplicities. Then
$$m = k \chi(V) - \chi(U)$$
where $\chi$ is the Euler characteristic.
\end{lemma}

See e.g. \cite[Theorem 12.47]{Zakeri2021}.

We can now proceed with the proof of Theorem \ref{answerq2}.

\begin{proof}
Let $r$ be a rational map of degree $k$ with $|r(\infty)|>1$.

First suppose that $L_r$ is an analytic Jordan curve. Then $\Omega_r:=\{z \in \hC:|r(z)|<1\}$ is simply connected. Moreover, we have that $r:\Omega_r \to \mD$ is a branched covering map. Using the fact that $\chi(\Omega_r)=\chi(\mD)=1$, the Riemann-Hurwitz formula (Lemma \ref{lemma:rh}) gives
$$m=k-1,$$
i.e. $r$ has precisely $k-1$ critical values in $\mD$ counting multiplicities. The remaining $k-1$ critical values of $r$ must belong to $\mD^*$ since $r$ cannot have critical points on $L_r$. This proves (i).

For (ii), note that we can simply take $\Gamma:=L_r$.
\\

Conversely, suppose that both conditions (i) and (ii) hold. Let $U$ denote the interior of the Jordan curve $\Gamma$. Then $|r| \leq 1$ on $\partial U = \Gamma$, hence $|r|<1$ on $U$ by the maximum modulus principle. It follows that $U \subset \Omega_r$. In fact $U$ has to be contained in a single component of $\Omega_r$, by connectedness. There can be no other component by Lemma \ref{lemma:zeros}, since $U$ contains all the zeros of $r$. This shows that $\Omega_r$ is connected, and we can apply the Riemann--Hurwitz formula to $r:\Omega_r \to \mD$ to obtain
$$k-1=k\chi(\mD) - \chi(\Omega_r)=k-\chi(\Omega_r),$$
since $m=k-1$ by (i). Hence $\chi(\Omega_r)=1$ and $\Omega_r$ is simply connected. Also, the rational map $r$ has no critical point on $L_r$ by (i). It follows from Lemma \ref{lemma:simplyconnected} that $L_r$ is an analytic Jordan curve, as required.
\end{proof}

\bibliographystyle{plain} 
\bibliography{biblio}

\begin{thebibliography}{10}

\bibitem{Bishop1987}
C.~J. Bishop.
\newblock {\em Harmonic measures supported on curves}.
\newblock ProQuest LLC, Ann Arbor, MI, 1987.
\newblock Thesis (Ph.D.)--The University of Chicago.

\bibitem{Bishop2002}
C.~J. Bishop.
\newblock Non-rectifiable limit sets of dimension one.
\newblock {\em Rev. Mat. Iberoamericana}, 18(3):653--684, 2002.

\bibitem{BishopCarlesonGarnettJones1989}
C.~J. Bishop, L.~Carleson, J.~B. Garnett, and P.~W. Jones.
\newblock Harmonic measures supported on curves.
\newblock {\em Pacific J. Math.}, 138(2):233--236, 1989.

\bibitem{BishopEremenkoLazebnik2025}
C.~J. Bishop, A.~Eremenko, and K.~Lazebnik.
\newblock On the {S}hapes of {R}ational {L}emniscates.
\newblock {\em Geom. Funct. Anal.}, 35(2):359--407, 2025.

\bibitem{BishopPeres2017}
C.~J. Bishop and Y.~Peres.
\newblock {\em Fractals in probability and analysis}, volume 162 of {\em
  Cambridge Studies in Advanced Mathematics}.
\newblock Cambridge University Press, Cambridge, 2017.

\bibitem{BrowderWermer1963}
A.~Browder and J.~Wermer.
\newblock Some algebras of functions on an arc.
\newblock {\em J. Math. Mech.}, 12:119--130, 1963.

\bibitem{BrowderWermer1964}
A.~Browder and J.~Wermer.
\newblock A method for constructing {D}irichlet algebras.
\newblock {\em Proc. Amer. Math. Soc.}, 15:546--552, 1964.

\bibitem{Conway1995}
J.~B. Conway.
\newblock {\em Functions of one complex variable. {II}}, volume 159 of {\em
  Graduate Texts in Mathematics}.
\newblock Springer-Verlag, New York, 1995.

\bibitem{Dubinin2008}
V.~N. Dubinin.
\newblock Majorization principles for meromorphic functions.
\newblock {\em Mat. Zametki}, 84(6):803--808, 2008.

\bibitem{EbenfeltKhavinsonShapiro2001}
P.~Ebenfelt, D.~Khavinson, and H.~S. Shapiro.
\newblock An inverse problem for the double layer potential.
\newblock {\em Comput. Methods Funct. Theory}, 1(2):387--401, 2001.

\bibitem{EbenfeltKhavinsonShapiro2011}
P.~Ebenfelt, D.~Khavinson, and H.~S. Shapiro.
\newblock Two-dimensional shapes and lemniscates.
\newblock In {\em Complex analysis and dynamical systems {IV}. {P}art 1},
  volume 553 of {\em Contemp. Math.}, pages 45--59. Amer. Math. Soc.,
  Providence, RI, 2011.

\bibitem{FortierBourqueYounsi2015}
M.~Fortier~Bourque and M.~Younsi.
\newblock Rational {A}hlfors functions.
\newblock {\em Constr. Approx.}, 41(1):157--183, 2015.

\bibitem{GarnettMarshall2005}
J.~B. Garnett and D.~E. Marshall.
\newblock {\em Harmonic measure}, volume~2 of {\em New Mathematical
  Monographs}.
\newblock Cambridge University Press, Cambridge, 2005.

\bibitem{HaymanLingham2019}
W.~K. Hayman and E.~F. Lingham.
\newblock {\em Research problems in function theory}.
\newblock Problem Books in Mathematics. Springer, Cham, anniversary edition,
  2019.

\bibitem{Khavinson2018}
D.~Khavinson.
\newblock Spectral properties of classical integral operators and geometry.
\newblock In {\em Spectral theory and applications}, volume 720 of {\em
  Contemp. Math.}, pages 205--212. Amer. Math. Soc., [Providence], RI, [2018]
  \copyright 2018.

\bibitem{LerarioLundberg2016}
Antonio Lerario and Erik Lundberg.
\newblock On the geometry of random lemniscates.
\newblock {\em Proc. Lond. Math. Soc. (3)}, 113(5):649--673, 2016.

\bibitem{Milnor1997}
J.~W. Milnor.
\newblock {\em Topology from the differentiable viewpoint}.
\newblock Princeton Landmarks in Mathematics. Princeton University Press,
  Princeton, NJ, 1997.
\newblock Based on notes by David W. Weaver, Revised reprint of the 1965
  original.

\bibitem{Osgood1903}
W.~F. Osgood.
\newblock A {J}ordan curve of positive area.
\newblock {\em Trans. Amer. Math. Soc.}, 4(1):107--112, 1903.

\bibitem{PouliasisRansford2016}
S.~Pouliasis and T.~Ransford.
\newblock On the harmonic measure and capacity of rational lemniscates.
\newblock {\em Potential Anal.}, 44(2):249--261, 2016.

\bibitem{Ransford1995}
T.~Ransford.
\newblock {\em Potential theory in the complex plane}, volume~28 of {\em London
  Mathematical Society Student Texts}.
\newblock Cambridge University Press, Cambridge, 1995.

\bibitem{Stein2005}
E.~M. Stein and R.~Shakarchi.
\newblock {\em Real analysis}, volume~3 of {\em Princeton Lectures in
  Analysis}.
\newblock Princeton University Press, Princeton, NJ, 2005.
\newblock Measure theory, integration, and Hilbert spaces.

\bibitem{WittViklund}
D.~Witt-Nystr\"{o}m and F.~Viklund.
\newblock Competitive hele-shaw flow and quadratic differentials.
\newblock preprint.

\bibitem{Younsi2016}
M.~Younsi.
\newblock Shapes, fingerprints and rational lemniscates.
\newblock {\em Proc. Amer. Math. Soc.}, 144(3):1087--1093, 2016.

\bibitem{Zakeri2021}
S.~Zakeri.
\newblock {\em A course in complex analysis}.
\newblock Princeton University Press, Princeton, NJ, 2021.

\end{thebibliography}

\end{document}